\documentclass[12pt,twoside]{amsart}
\usepackage{amsmath, amsthm, amscd, amsfonts, amssymb, graphicx}
\usepackage[bookmarksnumbered, plainpages]{hyperref}

\newtheorem{thm}{Theorem}[section]

\newtheorem{defn}[thm]{Definition}

\numberwithin{equation}{section}

\begin{document}

\title{\bf Algebraic Schouten Solitons of Three-Dimensional Lorentzian Lie Groups}
\author{Siyao Liu \hskip 0.4 true cm  Yong Wang$^{*}$}

\thanks{{\scriptsize
\hskip -0.4 true cm \textit{2010 Mathematics Subject Classification:}
53C40; 53C42.
\newline \textit{$^{*}$Corresponding author}}}

\maketitle

\begin{abstract}
In \cite{Wears}, Wears defined and studied algebraic T-solitons. In this paper, we give the definition of algebraic Schouten solitons as a special T-soliton and classify algebraic Schouten solitons associated to Levi-Civita connections, canonical connections and Kobayashi-Nomizu connections on three-dimensional Lorentzian Lie groups with some product structure.
\end{abstract}

\vskip 0.2 true cm


\pagestyle{myheadings}
\markboth{\rightline {\scriptsize Liu}}
         {\leftline{\scriptsize Algebraic Schouten Solitons of Three-Dimensional Lorentzian Lie Groups}}

\bigskip
\bigskip


\section{ Introduction}
\indent
Lauret introduced the Ricci soliton, which is a natural generalization of Einstein metric on nilpotent Lie groups, and introduced the algebraic Ricci soliton in Riemannian case in \cite{Lauret}.
Also, Lauret proved that algebraic Ricci solitons on homogeneous Riemannian manifolds are Ricci solitons.
The concept of the algebraic Ricci soliton was extended to the pseudo-Riemannian case by Onda in \cite{Onda}.
The author provided a study of algebraic Ricci solitons in the pseudo-Riemannian case and obtained a steady algebraic Ricci soliton and a shrinking algebraic Ricci soliton in the Lorentzian setting.
Note that in \cite{Batat}, Batat and Onda studied algebraic Ricci solitons of three-dimensional Lorentzian Lie groups, and they determined all three-dimensional Lorentzian Lie groups which are algebraic Ricci solitons.
Etayo and Santamaria studied some affine connections on manifolds with the product structure or the complex structure.
In particular, the canonical connection and the Kobayashi-Nomizu connection for a product structure were studied.
See \cite{Etayo} for details.
Wang defined algebraic Ricci solitons associated to canonical connections and Kobayashi-Nomizu connections in \cite{Wang1}.
And he classified algebraic Ricci solitons associated to canonical connections and Kobayashi-Nomizu connections on three-dimensional Lorentzian Lie groups with the product structure.
Other results of Ricci solitons are found in \cite{Brozos,Azami,Wang2,Wang3}.

Following the work of Lauret on the Ricci flow on nilpotent Lie groups, Wears defined algebraic T-solitons for the geometric evolution equation and established the relationship between algebraic T-solitons and T-solitons. In \cite{Wears}, the author showed that Lauret's ideas as for algebraic solitons apply equally well to an arbitrary geometric evolution equation (subjection to the appropriate conditions) for a left-invariant Riemannian metric on a simply connected Lie group.
In Eq.(1) \cite{Azami}, an generalized Ricci soliton was defined which could be considered as the Schouten soliton.

According to the generalization of the definition of the Schouten tensor in \cite{Louzao}, motivated by \cite{Azami,Wears}, we give the definition of algebraic Schouten solitons, which are Schouten solitons defined in \cite{Azami}.
In this paper, we investigate algebraic Schouten solitons associated to Levi-Civita connections, canonical connections and Kobayashi-Nomizu connections, and classify algebraic Schouten solitons associated to Levi-Civita connections, canonical connections and Kobayashi-Nomizu connections on three-dimensional Lorentzian Lie groups.

A brief description of the organization of this paper is as follows.
In Section 2, we recall the classification of three-dimensional Lorentzian Lie groups.
In Section 3, we classify algebraic Schouten solitons associated to Levi-Civita connections on three-dimensional Lorentzian Lie groups with the product structure.
In Section 4, we classify algebraic Schouten solitons associated to canonical connections and Kobayashi-Nomizu connections on three-dimensional Lorentzian Lie groups with the product structure.


\vskip 1 true cm
\section{ Three-dimensional unimodular Lorentzian Lie groups}

See \cite{Milnor}, Milnor gave a complete classification of three-dimensional unimodular Lie groups equipped with a left-invariant Riemannian metric.
In \cite{Rahmani}, Rahmani classified three-dimensional unimodular Lie groups equipped with a left-invariant Lorentzian metric.
Cordero and Parker wrote down the possible forms of a non-unimodular Lie algebra in \cite{Cordero}, which was proved by Calvaruso in \cite{Calvaruso}.
Three-dimensional Lorentzian Lie groups have been classified by the following theorems.

\begin{thm}
Let $(G, g)$ be a three-dimensional connected unimodular Lie group, equipped with a left-invariant Lorentzian metric. Then there exists a pseudo-orthonormal basis $\{e_{1}, e_{2}, e_{3}\}$ with $e_{3}$ timelike such that the Lie algebra of $G$ is one of the following:
\begin{align}
\notag
(\mathfrak{g}_{1}):\\ \notag
&[e_{1}, e_{2}]=\alpha e_{1}-\beta e_{3},~~[e_{1}, e_{3}]=-\alpha e_{1}-\beta e_{2},~~[e_{2}, e_{3}]=\beta e_{1}+\alpha e_{2}+\alpha e_{3},~~\alpha\neq 0.\\ \notag
(\mathfrak{g}_{2}):\\ \notag
&[e_{1}, e_{2}]=\gamma e_{2}-\beta e_{3},~~[e_{1}, e_{3}]=-\beta e_{2}-\gamma e_{3},~~[e_{2}, e_{3}]=\alpha e_{1},~~\gamma\neq 0.\\ \notag
(\mathfrak{g}_{3}):\\ \notag
&[e_{1}, e_{2}]=-\gamma e_{3},~~[e_{1}, e_{3}]=-\beta e_{2},~~[e_{2}, e_{3}]=\alpha e_1.\\ \notag
(\mathfrak{g}_{4}):\\ \notag
&[e_{1}, e_{2}]=-e_{2}+(2\eta-\beta)e_{3},~~\eta=1~{\rm or}-1,~~[e_{1}, e_{3}]=-\beta e_{2}+ e_{3},~~[e_{2}, e_{3}]=\alpha e_{1}. \notag
\end{align}
\end{thm}
\begin{thm}

Let $(G, g)$ be a three-dimensional connected non-unimodular Lie group, equipped with a left-invariant Lorentzian metric. Then there exists a pseudo-orthonormal basis $\{e_{1}, e_{2}, e_{3}\}$ with $e_{3}$ timelike such that the Lie algebra of $G$ is one of the following:
\begin{align}
\notag
(\mathfrak{g}_{5}):\\ \notag
&[e_{1}, e_{2}]=0,~~[e_{1}, e_{3}]=\alpha e_{1}+\beta e_{2},~~[e_{2}, e_{3}]=\gamma e_{1}+\delta e_{2},~~\alpha+\delta\neq 0,~~\alpha\gamma+\beta\delta=0.\\ \notag
(\mathfrak{g}_{6}):\\ \notag
&[e_{1}, e_{2}]=\alpha e_{2}+\beta e_{3},~~[e_{1}, e_{3}]=\gamma e_{2}+\delta e_{3},~~[e_{2}, e_{3}]=0,~~\alpha+\delta\neq 0,~~\alpha\gamma-\beta\delta=0.\\ \notag
(\mathfrak{g}_{7}):\\ \notag
&[e_{1}, e_{2}]=-\alpha e_{1}-\beta e_{2}-\beta e_{3},~~[e_{1}, e_{3}]=\alpha e_{1}+\beta e_2+\beta e_3,~~[e_{2}, e_{3}]=\gamma e_{1}+\delta e_{2}+\delta e_{3},\\ \notag
&~~\alpha+\delta\neq 0,~~\alpha\gamma=0. \notag
\end{align}
\end{thm}
\vskip 1 true cm

\section{ Algebraic Schouten Solitons associated to Levi-Civita connections on three-dimensional Lorentzian Lie groups}

Throughout this paper, we shall by $\{G_{i}\}_{i=1,\cdots,7}$, denote the connected, simply connected three-dimensional Lie group equipped with a left-invariant Lorentzian metric $g$ and having Lie algebra $\{\mathfrak{g}\}_{i=1,\cdots,7}$.
Let $\nabla$ be the Levi-Civita connection of $G_{i}$ and $R$ its curvature tensor, taken with the convention
\begin{equation}
R(X, Y)Z=\nabla_{X}\nabla_{Y}Z-\nabla_{Y}\nabla_{X}Z-\nabla_{[X, Y]}Z.
\end{equation}
The Ricci tensor of $(G_{i}, g)$ is defined by
\begin{equation}\rho(X, Y)=-g(R(X, e_{1})Y, e_{1})-g(R(X, e_{2})Y, e_{2})+g(R(X, e_{3})Y, e_{3}),
\end{equation}
where $\{e_{1}, e_{2}, e_{3}\}$ is a pseudo-orthonormal basis, with $e_{3}$ timelike and the Ricci operator Ric is given by
\begin{equation}
\rho(X, Y)=g({\rm Ric}(X), Y).
\end{equation}
The Schouten tensor is defined by
\begin{equation}
S(e_{i}, e_{j})=\rho(e_{i}, e_{j})-\frac{s}{4} g(e_{i}, e_{j}),
\end{equation}
where $s$ denotes the scalar curvature.
Generalized the definition of the Schouten tensor, we have
\begin{equation}
S(e_{i}, e_{j})=\rho(e_{i}, e_{j})-s \lambda_{0} g(e_{i}, e_{j}),
\end{equation}
where $\lambda_{0}$ is a real number.
Refer to \cite{Salimi}, we can get
\begin{equation}
s=\rho(e_{1}, e_{1})+\rho(e_{2}, e_{2})-\rho(e_{3}, e_{3}).
\end{equation}
\begin{defn}
$(G_{i}, g)$ is called the algebraic Schouten soliton associated to the connection $\nabla$ if it satisfies
\begin{equation}
{\rm Ric}=(s \lambda_{0}+c){\rm Id}+D,
\end{equation}
where $c$ is a real number, and $D$ is a derivation of $\mathfrak{g},$ that is
\begin{equation}
D[X, Y]=[DX, Y]+[X, DY]~~{\rm for }~ X, Y\in \mathfrak{g}.
\end{equation}
\end{defn}

\begin{thm}
$(G_{1}, g)$ is the algebraic Schouten soliton associated to the connection $\nabla$ if and only if $\beta=0,$ $c=0.$
\end{thm}
\begin{proof}
From \cite{Batat}, we have
\begin{align}
{\rm Ric}\left(\begin{array}{c}
e_{1}\\
e_{2}\\
e_{3}
\end{array}\right)=\left(\begin{array}{ccc}
\frac{1}{2}\beta^{2}&\alpha\beta&\alpha\beta\\
\alpha\beta&2\alpha^{2}+\frac{1}{2}\beta^{2}&2\alpha^{2}\\
-\alpha\beta&-2\alpha^{2}&-2\alpha^{2}+\frac{1}{2}\beta^{2}
\end{array}\right)\left(\begin{array}{c}
e_{1}\\
e_{2}\\
e_{3}
\end{array}\right).
\end{align}
Therefore, $s=\frac{3}{2}\beta^{2}.$
We can write down $D$ as
\begin{align}
\left\{\begin{array}{l}
De_{1}=(\frac{1}{2}\beta^{2}-\frac{3}{2}\beta^{2}\lambda_{0}-c)e_{1}+\alpha\beta e_{2}+\alpha\beta e_{3},\\
De_{2}=\alpha\beta e_{1}+(2\alpha^{2}+\frac{1}{2}\beta^{2}-\frac{3}{2}\beta^{2}\lambda_{0}-c)e_{2}+2\alpha^{2}e_{3},\\
De_{3}=-\alpha\beta e_{1}-2\alpha^{2} e_{2}+(-2\alpha^{2}+\frac{1}{2}\beta^{2}-\frac{3}{2}\beta^{2}\lambda_{0}-c)e_{3}.\\
\end{array}\right.
\end{align}
Hence, by Eq.(3.8) there exists a algebraic Schouten soliton associated to the connection $\nabla$ if and only if the following system of equations is satisfied
\begin{align}
\left\{\begin{array}{l}
\frac{3}{2}\alpha\beta^{2}\lambda_{0}+\alpha(\frac{3}{2}\beta^{2}+c)=0,\\
-\frac{3}{2}\beta^{3}\lambda_{0}+\beta(\frac{1}{2}\beta^{2}-c)=0,\\
\alpha\beta=0,\\
-\frac{3}{2}\beta^{3}\lambda_{0}+\beta(\frac{1}{2}\beta^{2}+6\alpha^{2}-c)=0,\\
-\frac{3}{2}\beta^{3}\lambda_{0}+\beta(\frac{1}{2}\beta^{2}-6\alpha^{2}-c)=0.\\
\end{array}\right.
\end{align}
Since $\alpha\neq 0,$ we get $\beta=0$ and $c=0.$
\end{proof}

\begin{thm}
$(G_{2}, g)$ is the algebraic Schouten soliton associated to the connection $\nabla$ if and only if $\alpha=\beta=0,$ $c=2\gamma^{2}(1-\lambda_{0}).$
\end{thm}
\begin{proof}
According to \cite{Batat}, we have
\begin{align}
{\rm Ric}\left(\begin{array}{c}
e_{1}\\
e_{2}\\
e_{3}
\end{array}\right)=\left(\begin{array}{ccc}
\frac{1}{2}\alpha^{2}+2\gamma^{2}&0&0\\
0&-\frac{1}{2}\alpha^{2}+\alpha\beta&-\alpha\gamma+2\beta\gamma\\
0&\alpha\gamma-2\beta\gamma&-\frac{1}{2}\alpha^{2}+\alpha\beta
\end{array}\right)\left(\begin{array}{c}
e_{1}\\
e_{2}\\
e_{3}
\end{array}\right).
\end{align}
Consequently, the scalar curvature is given by $s=-\frac{1}{2}\alpha^{2}+2\alpha\beta+2\gamma^{2}.$
We get
\begin{align}
\left\{\begin{array}{l}
De_{1}=(\frac{1}{2}\alpha^{2}+2\gamma^{2}-(-\frac{1}{2}\alpha^{2}+2\alpha\beta+2\gamma^{2})\lambda_{0}-c)e_{1},\\
De_{2}=(-\frac{1}{2}\alpha^{2}+\alpha\beta-(-\frac{1}{2}\alpha^{2}+2\alpha\beta+2\gamma^{2})\lambda_{0}-c)e_{2}+(-\alpha\gamma+2\beta\gamma)e_{3},\\
De_{3}=(\alpha\gamma-2\beta\gamma) e_{2}+(-\frac{1}{2}\alpha^{2}+\alpha\beta-(-\frac{1}{2}\alpha^{2}+2\alpha\beta+2\gamma^{2})\lambda_{0}-c)e_{3}.\\
\end{array}\right.
\end{align}
Eq.(3.8) is satisfied if and only if
\begin{align}
\left\{\begin{array}{l}
-(-\frac{1}{2}\alpha^{2}+2\alpha\beta+2\gamma^{2})\lambda_{0}+\frac{1}{2}\alpha^{2}+2\gamma^{2}-c=0,\\
\beta(-\frac{1}{2}\alpha^{2}+2\alpha\beta+2\gamma^{2})\lambda_{0}+2\gamma^{2}(\alpha-2\beta)-\beta(\frac{1}{2}\alpha^{2}+2\gamma^{2}-c)=0,\\
\alpha(-\frac{1}{2}\alpha^{2}+2\alpha\beta+2\gamma^{2})\lambda_{0}+\alpha(\frac{3}{2}\alpha^{2}+2\gamma^{2}-2\alpha\beta+c)=0.\\
\end{array}\right.
\end{align}
The first and second equations of the system Eq.(3.14) imply that
\begin{align}
\begin{array}{l}
(\alpha-2\beta)(\beta^{2}+\gamma^{2})=0.\\
\end{array}
\end{align}
Since, $\gamma\neq 0,$ we get $\alpha=2\beta.$
In this case, the system Eq.(3.14) reduces to
\begin{align}
\left\{\begin{array}{l}
-(\frac{1}{2}\alpha^{2}+2\gamma^{2})\lambda_{0}+\frac{1}{2}\alpha^{2}+2\gamma^{2}-c=0,\\
\alpha(\frac{1}{2}\alpha^{2}+2\gamma^{2})\lambda_{0}+\alpha(\frac{1}{2}\alpha^{2}+2\gamma^{2}+c)=0.\\
\end{array}\right.
\end{align}
If $\alpha=0,$ then we have $\beta=0,$ $c=2\gamma^{2}(1-\lambda_{0}).$
\end{proof}

\begin{thm}
$(G_{3}, g)$ is the algebraic Schouten soliton associated to the connection $\nabla$ if and only if\\
(i) $\alpha=\beta=\gamma=0,$ for all $c,$\\
(ii) $\alpha\neq 0,$ $\beta=\gamma=0,$ $c=-\frac{3}{2}\alpha^{2}+\frac{1}{2}\alpha^{2}\lambda_{0},$\\
(iii) $\alpha=\gamma=0,$ $\beta\neq 0,$ $c=-\frac{3}{2}\beta^{2}+\frac{1}{2}\beta^{2}\lambda_{0},$\\
(iv) $\alpha\neq 0,$ $\beta=\alpha,$ $\gamma=0,$ $c=0,$\\
(v) $\alpha\neq 0,$ $\beta=-\alpha,$ $\gamma=0,$ $c=-2\alpha^{2}+2\alpha^{2}\lambda_{0},$\\
(vi) $\alpha=\beta=0,$ $\gamma\neq 0,$ $c=-\frac{3}{2}\gamma^{2}+\frac{1}{2}\gamma^{2}\lambda_{0},$\\
(vii) $\alpha\neq 0,$ $\beta=0,$ $\gamma=\alpha,$ $c=0,$\\
(viii) $\alpha\neq 0,$ $\beta=0,$ $\gamma=\alpha,$ $c=-2\alpha^{2}+2\alpha^{2}\lambda_{0},$\\
(ix) $\alpha=0,$ $\beta\neq 0,$ $\gamma=\beta,$ $c=0,$\\
(x) $\alpha=0,$ $\beta\neq 0,$ $\gamma=-\beta,$ $c=-2\beta^{2}+2\beta^{2}\lambda_{0},$\\
(xi) $\alpha\neq 0,$ $\beta\neq 0,$ $\gamma=\alpha,$ $c=\frac{1}{2}\beta^{2}-(2\alpha\beta-\frac{1}{2}\beta^{2})\lambda_{0},$\\
(xii) $\alpha\neq 0,$ $\beta\neq 0,$ $\gamma=\beta-\alpha,$ $c=2\alpha\gamma-2\alpha\gamma\lambda_{0}.$
\end{thm}
\begin{proof}
By \cite{Batat}, we put
\begin{align}
a_{1}=\frac{1}{2}(\alpha-\beta-\gamma),~a_{2}=\frac{1}{2}(\alpha-\beta+\gamma),~a_{3}=\frac{1}{2}(\alpha+\beta-\gamma).
\end{align}
The Ricci operator is given by
\begin{align}
{\rm Ric}\left(\begin{array}{c}
e_{1}\\
e_{2}\\
e_{3}
\end{array}\right)=\left(\begin{array}{ccc}
l_{1}&0&0\\
0&l_{2}&0\\
0&0&l_{3}
\end{array}\right)\left(\begin{array}{c}
e_{1}\\
e_{2}\\
e_{3}
\end{array}\right),
\end{align}
where $l_{1}=a_{1}a_{2}+a_{1}a_{3}+\beta a_{2}+\gamma a_{3},$ $l_{2}=a_{1}a_{2}-a_{2}a_{3}-\alpha a_{1}+\gamma a_{3},$ $l_{3}=a_{1}a_{3}-a_{2}a_{3}-\alpha a_{1}+\beta a_{2}.$
And we have $s=2a_{1}a_{2}+2a_{1}a_{3}-2a_{2}a_{3}-2\alpha a_{1}+2\beta a_{2}+2\gamma a_{3}.$
So
\begin{align}
\left\{\begin{array}{l}
De_{1}=(l_{1}-(2a_{1}a_{2}+2a_{1}a_{3}-2a_{2}a_{3}-2\alpha a_{1}+2\beta a_{2}+2\gamma a_{3})\lambda_{0}-c)e_{1},\\
De_{2}=(l_{2}-(2a_{1}a_{2}+2a_{1}a_{3}-2a_{2}a_{3}-2\alpha a_{1}+2\beta a_{2}+2\gamma a_{3})\lambda_{0}-c)e_{2},\\
De_{3}=(l_{3}-(2a_{1}a_{2}+2a_{1}a_{3}-2a_{2}a_{3}-2\alpha a_{1}+2\beta a_{2}+2\gamma a_{3})\lambda_{0}-c)e_{3}.\\
\end{array}\right.
\end{align}
Therefore Eq.(3.8) now becomes
\begin{align}
\left\{\begin{array}{l}
\gamma(\frac{1}{2}\alpha^{2}+\frac{1}{2}\beta^{2}-\frac{3}{2}\gamma^{2}-\alpha\beta+\alpha\gamma+\beta\gamma-(-\frac{1}{2}\alpha^{2}-\frac{1}{2}\beta^{2}-\frac{1}{2}\gamma^{2}+\alpha\beta+\alpha\gamma+\beta\gamma)\lambda_{0}-c)=0,\\
\beta(\frac{1}{2}\alpha^{2}-\frac{3}{2}\beta^{2}+\frac{1}{2}\gamma^{2}+\alpha\beta-\alpha\gamma+\beta\gamma-(-\frac{1}{2}\alpha^{2}-\frac{1}{2}\beta^{2}-\frac{1}{2}\gamma^{2}+\alpha\beta+\alpha\gamma+\beta\gamma)\lambda_{0}-c)=0,\\
\alpha(\frac{3}{2}\alpha^{2}-\frac{1}{2}\beta^{2}-\frac{1}{2}\gamma^{2}-\alpha\beta-\alpha\gamma+\beta\gamma+(-\frac{1}{2}\alpha^{2}-\frac{1}{2}\beta^{2}-\frac{1}{2}\gamma^{2}+\alpha\beta+\alpha\gamma+\beta\gamma)\lambda_{0}+c)=0.\\
\end{array}\right.
\end{align}
Suppose that $\gamma=0,$ we get
\begin{align}
\left\{\begin{array}{l}
\beta(\frac{1}{2}\alpha^{2}-\frac{3}{2}\beta^{2}+\alpha\beta-(-\frac{1}{2}\alpha^{2}-\frac{1}{2}\beta^{2}+\alpha\beta)\lambda_{0}-c)=0,\\
\alpha(\frac{3}{2}\alpha^{2}-\frac{1}{2}\beta^{2}-\alpha\beta+(-\frac{1}{2}\alpha^{2}-\frac{1}{2}\beta^{2}+\alpha\beta)\lambda_{0}+c)=0.\\
\end{array}\right.
\end{align}
If $\beta=0,$ we obtain two cases (i)-(ii).
If $\beta\neq 0,$ for the cases (iii)-(v) the system Eq.(3.21) holds.
Now we assume that $\gamma\neq 0,$ then $c=\frac{1}{2}\alpha^{2}+\frac{1}{2}\beta^{2}-\frac{3}{2}\gamma^{2}-\alpha\beta+\alpha\gamma+\beta\gamma-(-\frac{1}{2}\alpha^{2}-\frac{1}{2}\beta^{2}-\frac{1}{2}\gamma^{2}+\alpha\beta+\alpha\gamma+\beta\gamma)\lambda_{0}.$
Meanwhile, we have
\begin{align}
\left\{\begin{array}{l}
\beta(-\beta^{2}+\gamma^{2}+\alpha\beta-\alpha\gamma)=0,\\
\alpha(\alpha^{2}-\gamma^{2}-\alpha\beta+\beta\gamma)=0.\\
\end{array}\right.
\end{align}
If $\beta=0,$  the cases (vi)-(viii) holds.
If $\beta\neq 0,$ for the cases (ix)-(xii) the system Eq.(3.22) holds.
\end{proof}

\begin{thm}
$(G_{4}, g)$ is the algebraic Schouten soliton associated to the connection $\nabla$ if and only if $\alpha=0,$ $\beta=\eta,$ $c=2\lambda_{0}.$
\end{thm}
\begin{proof}
\cite{Batat} makes it obvious that
\begin{align}
{\rm Ric}\left(\begin{array}{c}
e_{1}\\
e_{2}\\
e_{3}
\end{array}\right)=\left(\begin{array}{ccc}
\frac{1}{2}\alpha^{2}&0&0\\
0&-\frac{1}{2}\alpha^{2}+\alpha\beta-2\eta(\alpha-\beta)-2&\alpha-2\beta+2\eta\\
0&-\alpha+2\beta-2\eta&-\frac{1}{2}\alpha^{2}+\alpha\beta-2\beta\eta+2
\end{array}\right)\left(\begin{array}{c}
e_{1}\\
e_{2}\\
e_{3}
\end{array}\right).
\end{align}
A direct computation shows that the value of the scalar curvature is $-\frac{1}{2}\alpha^{2}+2\alpha\beta-2\alpha\eta-2.$
We get
\begin{align}
\left\{\begin{array}{l}
De_{1}=(\frac{1}{2}\alpha^{2}+(\frac{1}{2}\alpha^{2}-2\alpha\beta+2\alpha\eta+2)\lambda_{0}-c)e_{1},\\
De_{2}=(-\frac{1}{2}\alpha^{2}+\alpha\beta-2\eta(\alpha-\beta)-2+(\frac{1}{2}\alpha^{2}-2\alpha\beta+2\alpha\eta+2)\lambda_{0}-c)e_{2}+(\alpha-2\beta+2\eta)e_{3},\\
De_{3}=(-\alpha+2\beta-2\eta)e_{2}+(-\frac{1}{2}\alpha^{2}+\alpha\beta-2\beta\eta+2+(\frac{1}{2}\alpha^{2}-2\alpha\beta+2\alpha\eta+2)\lambda_{0}-c)e_{3}.\\
\end{array}\right.
\end{align}
By applying the formula shown in (3.8), we can calculate
\begin{align}
\left\{\begin{array}{l}
\frac{1}{2}\alpha^{2}+(\frac{1}{2}\alpha^{2}-2\alpha\beta+2\alpha\eta+2)\lambda_{0}-c+2(\beta-\eta)(\alpha-2(\beta-\eta))=0,\\
(2\eta-\beta)(\frac{1}{2}\alpha^{2}+(\frac{1}{2}\alpha^{2}-2\alpha\beta+2\alpha\eta+2)\lambda_{0}-c)+2\eta(\beta-\eta)(\alpha-2(\beta-\eta))=0,\\
\beta(\frac{1}{2}\alpha^{2}+(\frac{1}{2}\alpha^{2}-2\alpha\beta+2\alpha\eta+2)\lambda_{0}-c)+2\eta(\beta-\eta)(\alpha-2(\beta-\eta))=0,\\
\alpha(\frac{3}{2}\alpha^{2}+(-\frac{1}{2}\alpha^{2}+2\alpha\beta-2\alpha\eta-2)\lambda_{0}+c-2\alpha(\beta-\eta))=0.\\
\end{array}\right.
\end{align}
Via some simple calculations, we can obtain
\begin{align}
\left\{\begin{array}{l}
\frac{1}{2}\alpha^{2}+(\frac{1}{2}\alpha^{2}-2\alpha\beta+2\alpha\eta+2)\lambda_{0}-c+2(\beta-\eta)(\alpha-2(\beta-\eta))=0,\\
(\eta-\beta)(\frac{1}{2}\alpha^{2}+(\frac{1}{2}\alpha^{2}-2\alpha\beta+2\alpha\eta+2)\lambda_{0}-c)=0,\\
\alpha(\frac{3}{2}\alpha^{2}+(-\frac{1}{2}\alpha^{2}+2\alpha\beta-2\alpha\eta-2)\lambda_{0}+c-2\alpha(\beta-\eta))=0.\\
\end{array}\right.
\end{align}
Let $\beta=\eta,$ we get $\alpha=0,$ $c=2\lambda_{0}.$
\end{proof}

\begin{thm}
$(G_{5}, g)$ is the algebraic Schouten soliton associated to the connection $\nabla$ if and only if\\
(i) $\beta=\gamma=0,$ $c=-\alpha^{2}-\delta^{2}-(2\alpha^{2}+2\alpha\delta+2\delta^{2})\lambda_{0},$\\
(ii) $\beta\neq 0,$ $\gamma\neq 0,$ $\alpha^{2}+\beta^{2}=\gamma^{2}+\delta^{2},$ $c=-\alpha^{2}-\frac{1}{2}(\beta+\gamma)^{2}-\delta^{2}-(2\alpha^{2}+2\alpha\delta+\frac{1}{2}(\beta+\gamma)^{2}+2\delta^{2})\lambda_{0}.$
\end{thm}
\begin{proof}
By \cite{Batat}, it is immediate that
\begin{align}
{\rm Ric}\left(\begin{array}{c}
e_{1}\\
e_{2}\\
e_{3}
\end{array}\right)=\left(\begin{array}{ccc}
-\alpha^{2}-\alpha\delta-\frac{\beta^{2}-\gamma^{2}}{2}&0&0\\
0&-\alpha\delta+\frac{\beta^{2}-\gamma^{2}}{2}-\delta^{2}&0\\
0&0&-\alpha^{2}-\frac{(\beta+\gamma)^{2}}{2}-\delta^{2}
\end{array}\right)\left(\begin{array}{c}
e_{1}\\
e_{2}\\
e_{3}
\end{array}\right).
\end{align}
Then we have $s=-2\alpha^{2}-2\alpha\delta-\frac{1}{2}(\beta+\gamma)^{2}-2\delta^{2},$
and
\begin{align}
\left\{\begin{array}{l}
De_{1}=(-\alpha^{2}-\alpha\delta-\frac{1}{2}(\beta^{2}-\gamma^{2})-(-2\alpha^{2}-2\alpha\delta-\frac{1}{2}(\beta+\gamma)^{2}-2\delta^{2})\lambda_{0}-c)e_{1},\\
De_{2}=(-\alpha\delta+\frac{1}{2}(\beta^{2}-\gamma^{2})-\delta^{2}-(-2\alpha^{2}-2\alpha\delta-\frac{1}{2}(\beta+\gamma)^{2}-2\delta^{2})\lambda_{0}-c)e_{2},\\
De_{3}=(-\alpha^{2}-\frac{1}{2}(\beta+\gamma)^{2}-\delta^{2}-(-2\alpha^{2}-2\alpha\delta-\frac{1}{2}(\beta+\gamma)^{2}-2\delta^{2})\lambda_{0}-c)e_{3}.\\
\end{array}\right.
\end{align}
By using Eq.(3.8) and making tedious calculations, we have the following:
\begin{align}
\left\{\begin{array}{l}
(\alpha+\delta)(\alpha^{2}+\frac{1}{2}(\beta+\gamma)^{2}+\delta^{2}+(-2\alpha^{2}-2\alpha\delta-\frac{1}{2}(\beta+\gamma)^{2}-2\delta^{2})\lambda_{0}+c)=0,\\
\beta(2\alpha^{2}+\frac{1}{2}(3\beta^{2}+2\beta\gamma-\gamma^{2})+(-2\alpha^{2}-2\alpha\delta-\frac{1}{2}(\beta+\gamma)^{2}-2\delta^{2})\lambda_{0}+c)=0,\\
\gamma(-2\delta^{2}+\frac{1}{2}(\beta^{2}-2\beta\gamma-3\gamma^{2})-(-2\alpha^{2}-2\alpha\delta-\frac{1}{2}(\beta+\gamma)^{2}-2\delta^{2})\lambda_{0}-c)=0.\\
\end{array}\right.
\end{align}
We assume that $\beta=0.$
Since $\alpha+\delta\neq 0$ and $\alpha\gamma+\beta\delta=0,$ we get
\begin{align}
\left\{\begin{array}{l}
\alpha\gamma=0,\\
\alpha+\delta\neq 0,\\
(\alpha+\delta)(\alpha^{2}+\frac{1}{2}\gamma^{2}+\delta^{2}+(-2\alpha^{2}-2\alpha\delta-\frac{1}{2}\gamma^{2}-2\delta^{2})\lambda_{0}+c)=0,\\
\gamma(-2\delta^{2}-\frac{3}{2}\gamma^{2}-(-2\alpha^{2}-2\alpha\delta-\frac{1}{2}\gamma^{2}-2\delta^{2})\lambda_{0}-c)=0.\\
\end{array}\right.
\end{align}
Consider $\gamma=0,$ the case (i) is true.
Now, we assume that $\beta\neq 0,$ then $c=-2\alpha^{2}-\frac{1}{2}(3\beta^{2}+2\beta\gamma-\gamma^{2})-(-2\alpha^{2}-2\alpha\delta-\frac{1}{2}(\beta+\gamma)^{2}-2\delta^{2})\lambda_{0}.$
If $\gamma\neq 0,$ for the cases (ii) the system Eq.(3.29) holds.
\end{proof}

\begin{thm}
$(G_{6}, g)$ is the algebraic Schouten soliton associated to the connection $\nabla$ if and only if\\
(i) $\beta=\gamma=0,$ $c=\alpha^{2}+\delta^{2}-(2\alpha^{2}+2\delta^{2}+2\alpha\delta)\lambda_{0},$\\
(ii) $\alpha=\beta=0,$ $\gamma\neq 0,$ $\gamma^{2}=\delta^{2},$ $c=\frac{1}{2}\gamma^{2}-\frac{3}{2}\gamma^{2}\lambda_{0},$\\
(iii) $\alpha\neq 0,$ $\alpha^{2}=\beta^{2},$ $\gamma=\delta=0,$ $c=\frac{1}{2}\alpha^{2}-\frac{3}{2}\alpha^{2}\lambda_{0},$\\
(iv) $\beta\neq 0,$ $\gamma\neq 0,$ $\alpha^{2}-\beta^{2}=\delta^{2}-\gamma^{2},$ $c=\alpha^{2}-\frac{1}{2}(\beta-\gamma)^{2}+\delta^{2}-(2\alpha^{2}+2\alpha\delta-\frac{1}{2}(\beta-\gamma)^{2}+2\delta^{2})\lambda_{0}.$
\end{thm}
\begin{proof}
In \cite{Batat}, the Ricci operator is given by
\begin{align}
{\rm Ric}\left(\begin{array}{c}
e_{1}\\
e_{2}\\
e_{3}
\end{array}\right)=\left(\begin{array}{ccc}
\alpha^{2}-\frac{(\beta-\gamma)^{2}}{2}+\delta^{2}&0&0\\
0&\alpha^{2}+\alpha\delta-\frac{\beta^{2}-\gamma^{2}}{2}&0\\
0&0&\alpha\delta+\frac{\beta^{2}-\gamma^{2}}{2}+\delta^{2}
\end{array}\right)\left(\begin{array}{c}
e_{1}\\
e_{2}\\
e_{3}
\end{array}\right).
\end{align}
So $s=2\alpha^{2}+2\alpha\delta-\frac{1}{2}(\beta-\gamma)^{2}+2\delta^{2}.$
A simple calculation shows that
\begin{align}
\left\{\begin{array}{l}
De_{1}=(\alpha^{2}-\frac{1}{2}(\beta-\gamma)^{2}+\delta^{2}-(2\alpha^{2}+2\alpha\delta-\frac{1}{2}(\beta-\gamma)^{2}+2\delta^{2})\lambda_{0}-c)e_{1},\\
De_{2}=(\alpha^{2}+\alpha\delta-\frac{1}{2}(\beta^{2}-\gamma^{2})-(2\alpha^{2}+2\alpha\delta-\frac{1}{2}(\beta-\gamma)^{2}+2\delta^{2})\lambda_{0}-c)e_{2},\\
De_{3}=(\alpha\delta+\frac{1}{2}(\beta^{2}-\gamma^{2})+\delta^{2}-(2\alpha^{2}+2\alpha\delta-\frac{1}{2}(\beta-\gamma)^{2}+2\delta^{2})\lambda_{0}-c)e_{3}.\\
\end{array}\right.
\end{align}
Thus, equation (3.8) is satisfied if and only if
\begin{align}
\left\{\begin{array}{l}
(\alpha^{2}+\delta^{2})(-\alpha^{2}+\frac{1}{2}(\beta-\gamma)^{2}-\delta^{2}+(2\alpha^{2}+2\alpha\delta-\frac{1}{2}(\beta-\gamma)^{2}+2\delta^{2})\lambda_{0}+c)=0,\\
\beta(-2\alpha^{2}+\frac{1}{2}(3\beta^{2}-2\beta\gamma-\gamma^{2})+(2\alpha^{2}+2\alpha\delta-\frac{1}{2}(\beta-\gamma)^{2}+2\delta^{2})\lambda_{0}+c)=0,\\
\gamma(-2\delta^{2}+\frac{1}{2}(-\beta^{2}-2\beta\gamma+3\gamma^{2})+(2\alpha^{2}+2\alpha\delta-\frac{1}{2}(\beta-\gamma)^{2}+2\delta^{2})\lambda_{0}+c)=0.\\
\end{array}\right.
\end{align}
Suppose that $\beta=0,$ by taking into account $\alpha+\delta\neq 0$ and $\alpha\gamma+\beta\delta=0,$ we get
\begin{align}
\left\{\begin{array}{l}
\alpha\gamma=0,\\
\alpha+\delta\neq 0,\\
(\alpha^{2}+\delta^{2})(-\alpha^{2}+\frac{1}{2}\gamma^{2}-\delta^{2}+(2\alpha^{2}+2\alpha\delta-\frac{1}{2}\gamma^{2}+2\delta^{2})\lambda_{0}+c)=0,\\
\gamma(-2\delta^{2}+\frac{3}{2}\gamma^{2}+(2\alpha^{2}+2\alpha\delta-\frac{1}{2}\gamma^{2}+2\delta^{2})\lambda_{0}+c)=0.\\
\end{array}\right.
\end{align}
Set $\gamma=0,$ we obtain case (i).
If $\gamma\neq 0,$ we obtain case (ii).
Let $\beta\neq 0,$ then $c=2\alpha^{2}-\frac{1}{2}(3\beta^{2}-2\beta\gamma-\gamma^{2})-(2\alpha^{2}+2\alpha\delta-\frac{1}{2}(\beta-\gamma)^{2}+2\delta^{2})\lambda_{0}.$
Consequently,
\begin{align}
\left\{\begin{array}{l}
\alpha+\delta\neq 0,\\
\alpha\gamma+\beta\delta=0,\\
(\alpha^{2}+\delta^{2})(\alpha^{2}-\beta^{2}+\gamma^{2}-\delta^{2})=0,\\
\gamma(\alpha^{2}-\beta^{2}+\gamma^{2}-\delta^{2})=0.\\
\end{array}\right.
\end{align}
Consider $\gamma=0,$ then the case (iii) is true.
If $\gamma\neq 0,$ for the cases (iv) the system Eq.(3.33) holds.
\end{proof}

\begin{thm}
$(G_{7}, g)$ is the algebraic Schouten soliton associated to the connection $\nabla$ if and only if $\gamma=0,$ $c=0.$
\end{thm}
\begin{proof}
From \cite{Batat}, we have
\begin{align}
{\rm Ric}\left(\begin{array}{c}
e_{1}\\
e_{2}\\
e_{3}
\end{array}\right)=\left(\begin{array}{ccc}
\frac{1}{2}\gamma^{2}&0&0\\
0&\alpha^{2}-\alpha\delta+\beta\gamma-\frac{1}{2}\gamma^{2}&\alpha^{2}-\alpha\delta+\beta\gamma\\
0&-\alpha^{2}+\alpha\delta-\beta\gamma&-\alpha^{2}+\alpha\delta-\beta\gamma-\frac{1}{2}\gamma^{2}
\end{array}\right)\left(\begin{array}{c}
e_{1}\\
e_{2}\\
e_{3}
\end{array}\right).
\end{align}
Then $s=-\frac{1}{2}\gamma^{2}.$
Computations show that
\begin{align}
\left\{\begin{array}{l}
De_{1}=(\frac{1}{2}\gamma^{2}+\frac{1}{2}\gamma^{2}\lambda_{0}-c)e_{1},\\
De_{2}=(\alpha^{2}-\alpha\delta+\beta\gamma-\frac{1}{2}\gamma^{2}+\frac{1}{2}\gamma^{2}\lambda_{0}-c)e_{2}+(\alpha^{2}-\alpha\delta+\beta\gamma)e_{3},\\
De_{3}=(-\alpha^{2}+\alpha\delta-\beta\gamma)e_{2}+(-\alpha^{2}+\alpha\delta-\beta\gamma-\frac{1}{2}\gamma^{2}+\frac{1}{2}\gamma^{2}\lambda_{0}-c)e_{3}.\\
\end{array}\right.
\end{align}
Hence, Eq.(3.8) now yields
\begin{align}
\left\{\begin{array}{l}
(\alpha^{2}+\delta^{2})(\frac{1}{2}\gamma^{2}-\frac{1}{2}\gamma^{2}\lambda_{0}+c)=0,\\
\beta(\frac{1}{2}\gamma^{2}+\frac{1}{2}\gamma^{2}\lambda_{0}-c)=0,\\
\gamma(\frac{3}{2}\gamma^{2}-\frac{1}{2}\gamma^{2}\lambda_{0}+c)=0.\\
\end{array}\right.
\end{align}
Since $\alpha\gamma=0,$ $\alpha+\delta\neq 0,$ we get $\gamma=0$ and $c=0.$
\end{proof}
\vskip 1 true cm

\section{ Algebraic Schouten Solitons associated to canonical connections and Kobayashi-Nomizu connections on three-dimensional Lorentzian Lie groups}

We define a product structure $J$ on $G_{i}$ by
\begin{equation}
Je_{1}=e_{1},~Je_{2}=e_{2},~Je_{3}=-e_{3},
\end{equation}
then $J^2={\rm id}$ and $g(Je_{j}, Je_{j})=g(e_{j}, e_{j})$.
By \cite{Wang1}, we define the canonical connection and the Kobayashi-Nomizu connection as follows:
\begin{equation}
\nabla^{0}_{X}Y=\nabla_{X}Y-\frac{1}{2}(\nabla_{X}J)JY,
\end{equation}
\begin{equation}
\nabla^{1}_{X}Y=\nabla^{0}_{X}Y-\frac{1}{4}[(\nabla_{Y}J)JX-(\nabla_{JY}J)X].
\end{equation}
We define\begin{equation}
R^{0}(X, Y)Z=\nabla^{0}_{X}\nabla^{0}_{Y}Z-\nabla^{0}_{Y}\nabla^{0}_{X}Z-\nabla^{0}_{[X,Y]}Z,
\end{equation}
\begin{equation}
R^{1}(X, Y)Z=\nabla^{1}_{X}\nabla^{1}_{Y}Z-\nabla^{1}_{Y}\nabla^{1}_{X}Z-\nabla^{1}_{[X,Y]}Z.
\end{equation}
The Ricci tensors of $(G_i,g)$ associated to the canonical connection and the Kobayashi-Nomizu connection
are defined by
\begin{equation}
\rho^{0}(X, Y)=-g(R^{0}(X, e_{1})Y, e_{1})-g(R^{0}(X, e_{2})Y, e_{2})+g(R^{0}(X, e_{3})Y, e_{3}),
\end{equation}
\begin{equation}
\rho^{1}(X, Y)=-g(R^{1}(X, e_{1})Y, e_{1})-g(R^{1}(X, e_{2})Y, e_{2})+g(R^{1}(X, e_{3})Y, e_{3}).
\end{equation}
 The Ricci operators ${\rm Ric}^{0}$ and ${\rm Ric}^{1}$ is given by
\begin{equation}
\rho^{0}(X, Y)=g({\rm Ric}^{0}(X), Y),~~\rho^{1}(X, Y)=g({\rm Ric}^{1}(X), Y).
\end{equation}
Let
\begin{equation}
\widetilde{\rho}^{0}(X, Y)=\frac{{\rho}^{0}(X, Y)+{\rho}^{0}(Y, X)}{2},~~\widetilde{\rho}^{1}(X, Y)=\frac{{\rho}^{1}(X, Y)+{\rho}^{1}(Y, X)}{2},
\end{equation}
and
\begin{equation}
\widetilde{\rho}^{0}(X, Y)=g(\widetilde{{\rm Ric}}^{0}(X), Y),~~\widetilde{\rho}^{1}(X, Y)=g(\widetilde{{\rm Ric}}^{1}(X), Y).
\end{equation}
Similar to the formulae (3.5) and (3.6), we have
\begin{equation}
S^{0}(e_{i}, e_{j})=\widetilde{\rho}^{0}(e_{i}, e_{j})-s^{0} \lambda_{0} g(e_{i}, e_{j}),~~S^{1}(e_{i}, e_{j})=\widetilde{\rho}^{1}(e_{i}, e_{j})-s^{1} \lambda_{0} g(e_{i}, e_{j}),
\end{equation}
and
\begin{equation}
s^{0}=\widetilde{\rho}^{0}(e_{1}, e_{1})+\widetilde{\rho}^{0}(e_{2}, e_{2})-\widetilde{\rho}^{0}(e_{3}, e_{3}),~~s^{1}=\widetilde{\rho}^{1}(e_{1}, e_{1})+\widetilde{\rho}^{1}(e_{2}, e_{2})-\widetilde{\rho}^{1}(e_{3}, e_{3}).
\end{equation}

\begin{defn}
$(G_{i}, g, J)$ is called the algebraic Schouten soliton associated to the connection $\nabla^{0}$ if it satisfies
\begin{equation}
\widetilde{{\rm Ric}}^{0}=(s^{0} \lambda_{0}+c){\rm Id}+D,
\end{equation}
where $c$ is a real number, and $D$ is a derivation of $\mathfrak{g}$, that is
\begin{equation}
D[X, Y]=[DX, Y]+[X, DY]~~{\rm for }~ X, Y\in \mathfrak{g}.
\end{equation}
$(G_i,g,J)$ is called the algebraic Schouten soliton associated to the connection $\nabla^{1}$ if it satisfies
\begin{equation}
\widetilde{{\rm Ric}}^{1}=(s^{1} \lambda_{0}+c){\rm Id}+D.
\end{equation}
\end{defn}
\vskip 0.5 true cm

\begin{thm}
$(G_{1}, g, J)$ is the algebraic Schouten soliton associated to the connection $\nabla^{0}$ if and only if $\beta=0,$ $c=-\frac{1}{2}\alpha^{2}+2\alpha^{2}\lambda_{0}.$
\end{thm}
\begin{proof}
From \cite{Azami}, it is obvious that
\begin{align}
\widetilde{{\rm Ric}}^{0}\left(\begin{array}{c}
e_{1}\\
e_{2}\\
e_{3}
\end{array}\right)=\left(\begin{array}{ccc}
-(\alpha^{2}+\frac{1}{2}\beta^{2})&0&-\frac{1}{4}\alpha\beta\\
0&-(\alpha^{2}+\frac{1}{2}\beta^{2})&-\frac{1}{2}\alpha^{2}\\
\frac{1}{4}\alpha\beta&\frac{1}{2}\alpha^{2}&0
\end{array}\right)\left(\begin{array}{c}
e_{1}\\
e_{2}\\
e_{3}
\end{array}\right).
\end{align}
And $s^{0}=-2(\alpha^{2}+\frac{1}{2}\beta^{2}).$
We obtain that
\begin{align}
\left\{\begin{array}{l}
De_{1}=-(\alpha^{2}+\frac{1}{2}\beta^{2}-2(\alpha^{2}+\frac{1}{2}\beta^{2})\lambda_{0}+c)e_{1}-\frac{1}{4}\alpha\beta e_{3},\\
De_{2}=-(\alpha^{2}+\frac{1}{2}\beta^{2}-2(\alpha^{2}+\frac{1}{2}\beta^{2})\lambda_{0}+c)e_{2}-\frac{1}{2}\alpha^{2}e_{3},\\
De_{3}=\frac{1}{4}\alpha\beta e_{1}+\frac{1}{2}\alpha^{2}e_{2}-(-2(\alpha^{2}+\frac{1}{2}\beta^{2})\lambda_{0}+c)e_{3}.\\
\end{array}\right.
\end{align}
Then, Equation (4.14) becomes
\begin{align}
\left\{\begin{array}{l}
\alpha^{2}\beta=0,\\
\alpha(\frac{1}{2}\alpha^{2}-2(\alpha^{2}+\frac{1}{2}\beta^{2})\lambda_{0}+c)=0,\\
\beta(\frac{5}{2}\alpha^{2}+\beta^{2}-2(\alpha^{2}+\frac{1}{2}\beta^{2})\lambda_{0}+c)=0,\\
\beta(-2(\alpha^{2}+\frac{1}{2}\beta^{2})\lambda_{0}+c)=0,\\
\beta(\frac{1}{2}\alpha^{2}-2(\alpha^{2}+\frac{1}{2}\beta^{2})\lambda_{0}+c)=0.\\
\end{array}\right.
\end{align}
Taking into account that, $\alpha\neq0,$ we get $\beta=0$ and $c=-\frac{1}{2}\alpha^{2}+2\alpha^{2}\lambda_{0}.$
\end{proof}

\begin{thm}
$(G_{1}, g, J)$ is the algebraic Schouten soliton associated to the connection $\nabla^{1}$ if and only if $\beta=0,$ $c=-\frac{1}{2}\alpha^{2}+2\alpha^{2}\lambda_{0}.$
\end{thm}
\begin{proof}
In \cite{Azami}, it is shown that
\begin{align}
\widetilde{{\rm Ric}}^{1}\left(\begin{array}{c}
e_{1}\\
e_{2}\\
e_{3}
\end{array}\right)=\left(\begin{array}{ccc}
-(\alpha^{2}+\beta^{2})&\alpha\beta&\frac{1}{2}\alpha\beta\\
\alpha\beta&-(\alpha^{2}+\beta^{2})&-\frac{1}{2}\alpha^{2}\\
-\frac{1}{2}\alpha\beta&\frac{1}{2}\alpha^{2}&0
\end{array}\right)\left(\begin{array}{c}
e_{1}\\
e_{2}\\
e_{3}
\end{array}\right).
\end{align}
Therefore $s^{1}=-2(\alpha^{2}+\beta^{2}).$
$D$ is described by
\begin{align}
\left\{\begin{array}{l}
De_{1}=-(\alpha^{2}+\beta^{2}-2(\alpha^{2}+\beta^{2})\lambda_{0}+c)e_{1}+\alpha\beta e_{2}+\frac{1}{2}\alpha\beta e_{3},\\
De_{2}=\alpha\beta e_{1}-(\alpha^{2}+\beta^{2}-2(\alpha^{2}+\beta^{2})\lambda_{0}+c)e_{2}-\frac{1}{2}\alpha^{2}e_{3},\\
De_{3}=-\frac{1}{2}\alpha\beta e_{1}+\frac{1}{2}\alpha^{2}e_{2}-(-2(\alpha^{2}+\beta^{2})\lambda_{0}+c)e_{3}.\\
\end{array}\right.
\end{align}
We calculate that
\begin{align}
\left\{\begin{array}{l}
\alpha^{2}\beta=0,\\
\alpha(\frac{1}{2}\alpha^{2}+2\beta^{2}-2(\alpha^{2}+\beta^{2})\lambda_{0}+c)=0,\\
\beta(\alpha^{2}+2\beta^{2}-2(\alpha^{2}+\beta^{2})\lambda_{0}+c)=0,\\
\beta(\alpha^{2}-2(\alpha^{2}+\beta^{2})\lambda_{0}+c)=0.\\
\end{array}\right.
\end{align}
Note that, $\alpha\neq0,$ we get $\beta=0$ and $c=-\frac{1}{2}\alpha^{2}+2\alpha^{2}\lambda_{0}.$
\end{proof}

\begin{thm}
$(G_{2}, g, J)$ is the algebraic Schouten soliton associated to the connection $\nabla^{0}$ if and only if $\alpha=\beta=0,$ $c=-\gamma^{2}+2\gamma^{2}\lambda_{0}.$
\end{thm}
\begin{proof}
According to \cite{Azami}, we have
\begin{align}
\widetilde{{\rm Ric}}^{0}\left(\begin{array}{c}
e_{1}\\
e_{2}\\
e_{3}
\end{array}\right)=\left(\begin{array}{ccc}
-(\frac{1}{2}\alpha\beta+\gamma^2)&0&0\\
0&-(\frac{1}{2}\alpha\beta+\gamma^2)&\frac{1}{4}\alpha\gamma-\frac{1}{2}\beta\gamma\\
0&-\frac{1}{4}\alpha\gamma+\frac{1}{2}\beta\gamma&0
\end{array}\right)\left(\begin{array}{c}
e_{1}\\
e_{2}\\
e_{3}
\end{array}\right).
\end{align}
Obviously, $s^{0}=-(\alpha\beta+2\gamma^2).$
From Equation (4.13) it is easy to get that
\begin{align}
\left\{\begin{array}{l}
De_{1}=-(\frac{1}{2}\alpha\beta+\gamma^2-(\alpha\beta+2\gamma^2)\lambda_{0}+c)e_{1},\\
De_{2}=-(\frac{1}{2}\alpha\beta+\gamma^2-(\alpha\beta+2\gamma^2)\lambda_{0}+c)e_{2}+(\frac{1}{4}\alpha\gamma-\frac{1}{2}\beta\gamma)e_{3},\\
De_{3}=(-\frac{1}{4}\alpha\gamma+\frac{1}{2}\beta\gamma)e_{2}-(-(\alpha\beta+2\gamma^2)\lambda_{0}+c)e_{3}.\\
\end{array}\right.
\end{align}
Consequently, we obtain
\begin{align}
\left\{\begin{array}{l}
\gamma(\alpha\beta-\beta^{2}+\gamma^{2}-(\alpha\beta+2\gamma^2)\lambda_{0}+c)=0,\\
\beta(\alpha\beta+2\gamma^{2}-(\alpha\beta+2\gamma^2)\lambda_{0}+c)+\gamma(-\frac{1}{2}\alpha\gamma+\beta\gamma)=0,\\
\beta(-(\alpha\beta+2\gamma^2)\lambda_{0}+c)+\gamma(-\frac{1}{2}\alpha\gamma+\beta\gamma)=0,\\
\alpha(-(\alpha\beta+2\gamma^2)\lambda_{0}+c)=0.\\
\end{array}\right.
\end{align}
The second and third equations in (4.24) transforms into
\begin{align}
\beta(\alpha\beta+2\gamma^{2})=0.
\end{align}
Then, we get
\begin{align}
\left\{\begin{array}{l}
\gamma(\alpha^{2}+\alpha\beta-\beta^{2}-(\alpha\beta+2\gamma^2)\lambda_{0}+c)=0,\\
\beta(\alpha\beta+2\gamma^{2})=0,\\
\alpha(-(\alpha\beta+2\gamma^2)\lambda_{0}+c)=0.\\
\end{array}\right.
\end{align}
Note that, $\gamma\neq0.$
If $\alpha=0,$ then $\beta=0,$ $c=-\gamma^{2}+2\gamma^{2}\lambda_{0}.$
\end{proof}

\begin{thm}
$(G_{2}, g, J)$ is the algebraic Schouten soliton associated to the connection $\nabla^{1}$ if and only if $\alpha=\beta=0,$ $c=-\gamma^{2}+2\gamma^{2}\lambda_{0}.$
\end{thm}
\begin{proof}
We have
\begin{align}
\widetilde{{\rm Ric}}^{1}\left(\begin{array}{c}
e_{1}\\
e_{2}\\
e_{3}
\end{array}\right)=\left(\begin{array}{ccc}
-(\beta^{2}+\gamma^{2})&0&0\\
0&-(\alpha\beta+\gamma^{2})&\frac{1}{2}\alpha\gamma\\
0&-\frac{1}{2}\alpha\gamma&0
\end{array}\right)\left(\begin{array}{c}
e_{1}\\
e_{2}\\
e_{3}
\end{array}\right),
\end{align}
this can be found in \cite{Azami}.
And $s^{1}=-(\alpha\beta+\beta^{2}+2\gamma^{2}).$
From this, $D$ is given by
\begin{align}
\left\{\begin{array}{l}
De_{1}=-(\beta^{2}+\gamma^{2}-(\alpha\beta+\beta^{2}+2\gamma^{2})\lambda_{0}+c)e_{1},\\
De_{2}=-(\alpha\beta+\gamma^{2}-(\alpha\beta+\beta^{2}+2\gamma^{2})\lambda_{0}+c)e_{2}+\frac{1}{2}\alpha\gamma e_{3},\\
De_{3}=-\frac{1}{2}\alpha\gamma e_{2}-(-(\alpha\beta+\beta^{2}+2\gamma^{2})\lambda_{0}+c)e_{3}.\\
\end{array}\right.
\end{align}
In this way, Eq.(4.14) is satisfied if and only if
\begin{align}
\left\{\begin{array}{l}
\gamma(\alpha\beta+\beta^{2}+\gamma^{2}-(\alpha\beta+\beta^{2}+2\gamma^{2})\lambda_{0}+c)=0,\\
\beta(\alpha\beta+\beta^{2}+2\gamma^{2}-(\alpha\beta+\beta^{2}+2\gamma^{2})\lambda_{0}+c)-\alpha\gamma^{2}=0,\\
\beta(-\alpha\beta+\beta^{2}-(\alpha\beta+\beta^{2}+2\gamma^{2})\lambda_{0}+c)-\alpha\gamma^{2}=0,\\
\alpha(\alpha\beta-\beta^{2}-(\alpha\beta+\beta^{2}+2\gamma^{2})\lambda_{0}+c)=0.\\
\end{array}\right.
\end{align}
Since $\gamma\neq0,$ we get $c=-\alpha\beta-\beta^{2}-\gamma^{2}+(\alpha\beta+\beta^{2}+2\gamma^{2})\lambda_{0}.$
The second equation in (4.29) transforms into
\begin{align}
(\alpha-\beta)\gamma^{2}=0.
\end{align}
We have $\alpha=\beta=0,$ $c=-\gamma^{2}+2\gamma^{2}\lambda_{0}.$
\end{proof}

\begin{thm}
$(G_{3}, g, J)$ is the algebraic Schouten soliton associated to the connection $\nabla^{0}$ if and only if\\
(i) $\alpha=\beta=\gamma=0,$ for all $c,$\\
(ii) $\alpha=\beta=0,$ $\gamma\neq 0,$ $c=\gamma^{2}-\gamma^{2}\lambda_{0},$\\
(iii) $\alpha\neq 0$ or $\beta\neq 0,$ $\gamma=0,$ $c=0,$\\
(iv) $\alpha\neq 0$ or $\beta\neq 0,$ $\gamma=\alpha+\beta,$ $c=0.$
\end{thm}
\begin{proof}
By \cite{Azami}, we get
\begin{align}
\widetilde{{\rm Ric}}^{0}\left(\begin{array}{c}
e_{1}\\
e_{2}\\
e_{3}
\end{array}\right)=\left(\begin{array}{ccc}
-\gamma a_{3}&0&0\\
0&-\gamma a_{3}&0\\
0&0&0
\end{array}\right)\left(\begin{array}{c}
e_{1}\\
e_{2}\\
e_{3}
\end{array}\right),
\end{align}
where
\begin{equation}
a_{1}=\frac{1}{2}(\alpha-\beta-\gamma),~a_{2}=\frac{1}{2}(\alpha-\beta+\gamma),~a_{3}=\frac{1}{2}(\alpha+\beta-\gamma).
\end{equation}
A direct computation for the scalar curvature shows that $s^{0}=-2\gamma a_{3}=-\gamma(\alpha+\beta-\gamma).$
It is easy to obtain that
\begin{align}
\left\{\begin{array}{l}
De_{1}=-(\gamma a_{3}-2\gamma a_{3}\lambda_{0}+c)e_{1},\\
De_{2}=-(\gamma a_{3}-2\gamma a_{3}\lambda_{0}+c)e_{2},\\
De_{3}=-(-2\gamma a_{3}\lambda_{0}+c)e_{3}.\\
\end{array}\right.
\end{align}
Thus,
\begin{align}
\left\{\begin{array}{l}
\gamma(\gamma(\alpha+\beta-\gamma)-\gamma(\alpha+\beta-\gamma)\lambda_{0}+c)=0,\\
\beta(-\gamma(\alpha+\beta-\gamma)\lambda_{0}+c)=0,\\
\alpha(-\gamma(\alpha+\beta-\gamma)\lambda_{0}+c)=0.\\
\end{array}\right.
\end{align}
If $\alpha=0,$ then the cases (i)-(iii) holds.
Choose $\alpha\neq0$ and $c=\gamma(\alpha+\beta-\gamma)\lambda_{0}$, we obtain two cases (iii)-(iv).
\end{proof}

\begin{thm}
$(G_{3}, g, J)$ is the algebraic Schouten soliton associated to the connection $\nabla^{1}$ if and only if\\
(i) $\alpha=\beta=\gamma=0,$ $c\neq 0,$\\
(ii) $\alpha=0,$ $c=-\beta\gamma+\beta\gamma\lambda_{0},$\\
(iii) $\beta= 0,$ $c=-\alpha\gamma+\alpha\gamma\lambda_{0},$\\
(iv) $\alpha\beta\neq 0,$ $\gamma=0,$ $c=0.$
\end{thm}
\begin{proof}
We have
\begin{align}
\widetilde{{\rm Ric}}^{1}\left(\begin{array}{c}
e_{1}\\
e_{2}\\
e_{3}
\end{array}\right)=\left(\begin{array}{ccc}
\gamma(a_{1}-a_{3})&0&0\\
0&-\gamma(a_{2}+a_{3})&0\\
0&0&0
\end{array}\right)\left(\begin{array}{c}
e_{1}\\
e_{2}\\
e_{3}
\end{array}\right),
\end{align}
which is clear from \cite{Azami}.
By definition, we get $s^{1}=\gamma(a_{1}-a_{2}-2a_{3})=-\gamma(\alpha+\beta).$
Hence,
\begin{align}
\left\{\begin{array}{l}
De_{1}=-(-\gamma(a_{1}-a_{3})+\gamma(a_{1}-a_{2}-2a_{3})\lambda_{0}+c)e_{1},\\
De_{2}=-(\gamma(a_{2}+a_{3})+\gamma(a_{1}-a_{2}-2a_{3})\lambda_{0}+c)e_{2},\\
De_{3}=-(\gamma(a_{1}-a_{2}-2a_{3})\lambda_{0}+c)e_{3}.\\
\end{array}\right.
\end{align}
Equation (4.14) now becomes
\begin{align}
\left\{\begin{array}{l}
\gamma(\alpha\gamma+\beta\gamma-\gamma(\alpha+\beta)\lambda_{0}+c)=0,\\
\beta(-\alpha\gamma+\beta\gamma-\gamma(\alpha+\beta)\lambda_{0}+c)=0,\\
\alpha(\alpha\gamma-\beta\gamma-\gamma(\alpha+\beta)\lambda_{0}+c)=0.\\
\end{array}\right.
\end{align}
It is easy check that
\begin{align}
\left\{\begin{array}{l}
\alpha\beta\gamma^{2}=0,\\
\alpha\beta(-\gamma(\alpha+\beta)\lambda_{0}+c)=0.\\
\end{array}\right.
\end{align}
We consider $\alpha\beta=0,$ the case (i)-(iii) is holds.
If we consider $\alpha\beta\neq 0,$ then $\gamma=0,$ $c=0$ and the case (iv) holds.
\end{proof}

\begin{thm}
$(G_{4}, g, J)$ is the algebraic Schouten soliton associated to the connection $\nabla^{0}$ if and only if $\alpha=0,$ $\beta=\eta,$ $c=0.$
\end{thm}
\begin{proof}
From \cite{Azami}, we have
\begin{align}
\widetilde{{\rm Ric}}^{0}\left(\begin{array}{c}
e_{1}\\
e_{2}\\
e_{3}
\end{array}\right)=\left(\begin{array}{ccc}
b_{3}(2\eta-\beta)-1&0&0\\
0&b_{3}(2\eta-\beta)-1&-\frac{1}{2}(b_{3}-\beta)\\
0&\frac{1}{2}(b_{3}-\beta)&0
\end{array}\right)\left(\begin{array}{c}
e_{1}\\
e_{2}\\
e_{3}
\end{array}\right),
\end{align}
where
\begin{equation}
b_1=\frac{1}{2}\alpha+\eta-\beta,~b_2=\frac{1}{2}\alpha-\eta,~b_3=\frac{1}{2}\alpha+\eta.
\end{equation}
Then $s^{0}=2b_{3}(2\eta-\beta)-2=(2\eta+\alpha)(2\eta-\beta)-2.$
According to the condition $\widetilde{{\rm Ric}}^{0}=(s^{0} \lambda_{0}+c){\rm Id}+D,$ we calculate that
\begin{align}
\left\{\begin{array}{l}
De_{1}=(b_{3}(2\eta-\beta)-1-2b_{3}(2\eta-\beta)\lambda_{0}+2\lambda_{0}-c)e_{1},\\
De_{2}=(b_{3}(2\eta-\beta)-1-2b_{3}(2\eta-\beta)\lambda_{0}+2\lambda_{0}-c)e_{2}-\frac{1}{2}(b_{3}-\beta)e_{3},\\
De_{3}=\frac{1}{2}(b_{3}-\beta)e_{2}-(2b_{3}(2\eta-\beta)\lambda_{0}-2\lambda_{0}+c)e_{3}.\\
\end{array}\right.
\end{align}
Hence, Eq.(4.14) now yields
\begin{align}
\left\{\begin{array}{l}
\alpha((2\eta+\alpha)(2\eta-\beta)\lambda_{0}-2\lambda_{0}+c)=0,\\
\beta((2\eta+\alpha)(2\eta-\beta)\lambda_{0}-2\lambda_{0}+c)-(\frac{1}{2}\alpha+\eta-\beta)=0,\\
(2\eta-\beta)((2\eta+\alpha)(2\eta-\beta)-2-(2\eta+\alpha)(2\eta-\beta)\lambda_{0}+2\lambda_{0}-c)-(\frac{1}{2}\alpha+\eta-\beta)=0,\\
(\frac{1}{2}\alpha+\eta)(2\eta-\beta)-1-(2\eta+\alpha)(2\eta-\beta)\lambda_{0}+2\lambda_{0}-c+(\frac{1}{2}\alpha+\eta-\beta)(\eta-\beta)=0.\\
\end{array}\right.
\end{align}
For $\eta=\pm1$ and $\alpha=0,$ a straightforward calculation shows that
\begin{align}
\left\{\begin{array}{l}
\beta(-2\beta\eta\lambda_{0}+2\lambda_{0}+c)-(\eta-\beta)=0,\\
(2\eta-\beta)(-2\beta\eta+2+2\beta\eta\lambda_{0}-2\lambda_{0}-c)-(\eta-\beta)=0,\\
-\beta\eta+1+2\beta\eta\lambda_{0}-2\lambda_{0}-c+(\eta-\beta)^{2}=0.\\
\end{array}\right.
\end{align}
Solving (4.43), we get $\beta=\eta,$ $c=0.$
\end{proof}

\begin{thm}
$(G_{4}, g, J)$ is not the algebraic Schouten soliton associated to the connection $\nabla^{1}.$
\end{thm}
\begin{proof}
In this case we have
\begin{align}
\widetilde{{\rm Ric}}^{1}\left(\begin{array}{c}
e_{1}\\
e_{2}\\
e_{3}
\end{array}\right)=\left(\begin{array}{ccc}
-(1+(\beta-2\eta)(b_{3}-b_{1}))&0&0\\
0&-(1+(\beta-2\eta)(b_{2}+b_{3}))&\frac{b_{1}-b_{3}-\alpha+\beta}{2}\\
0&\frac{\alpha-\beta-b_{1}+b_{3}}{2}&0
\end{array}\right)\left(\begin{array}{c}
e_{1}\\
e_{2}\\
e_{3}
\end{array}\right).
\end{align}
That is
\begin{align}
\widetilde{{\rm Ric}}^{1}\left(\begin{array}{c}
e_{1}\\
e_{2}\\
e_{3}
\end{array}\right)=\left(\begin{array}{ccc}
-(1+\beta(\beta-2\eta))&0&0\\
0&-(1+\alpha(\beta-2\eta))&-\frac{1}{2}\alpha\\
0&\frac{1}{2}\alpha&0
\end{array}\right)\left(\begin{array}{c}
e_{1}\\
e_{2}\\
e_{3}
\end{array}\right).
\end{align}
So we have $s^{1}=-(2+(\alpha+\beta)(\beta-2\eta)).$
If $(G_{4}, g, J)$ is the algebraic Schouten soliton associated to the connection $\nabla^1$, then $\widetilde{{\rm Ric}}^{1}=(s^{1} \lambda_{0}+c){\rm Id}+D,$ so
\begin{align}
\left\{\begin{array}{l}
De_{1}=-(1+\beta(\beta-2\eta)-(2+(\alpha+\beta)(\beta-2\eta))\lambda_{0}+c)e_{1},\\
De_{2}=-(1+\alpha(\beta-2\eta)-(2+(\alpha+\beta)(\beta-2\eta))\lambda_{0}+c)e_{2}-\frac{1}{2}\alpha e_{3},\\
De_{3}=\frac{1}{2}\alpha e_{2}-(-(2+(\alpha+\beta)(\beta-2\eta))\lambda_{0}+c)e_{3}.\\
\end{array}\right.
\end{align}
For this reason Equation (4.14) now becomes
\begin{align}
\left\{\begin{array}{l}
1+(\frac{1}{2}\alpha+\beta)(\beta-2\eta)-(2+(\alpha+\beta)(\beta-2\eta))\lambda_{0}+c+\frac{1}{2}\alpha\beta=0,\\
(\beta-2\eta)(2+(\alpha+\beta)(\beta-2\eta)-(2+(\alpha+\beta)(\beta-2\eta))\lambda_{0}+c)-\alpha=0,\\
\beta(-(\alpha-\beta)(\beta-2\eta)-(2+(\alpha+\beta)(\beta-2\eta))\lambda_{0}+c)-\alpha=0,\\
\alpha((\alpha-\beta)(\beta-2\eta)-(2+(\alpha+\beta)(\beta-2\eta))\lambda_{0}+c)=0.\\
\end{array}\right.
\end{align}
(4.47) has no solutions, we find that $(G_{4}, g, J)$ is not the algebraic Schouten soliton associated to the connection $\nabla^{1}.$
\end{proof}

\begin{thm}
$(G_{5}, g, J)$ is the algebraic Schouten soliton associated to the connection $\nabla^{0}$ if and only if $c=0.$
\end{thm}
\begin{proof}
We have
\begin{align}
\widetilde{{\rm Ric}}^{0}\left(\begin{array}{c}
e_{1}\\
e_{2}\\
e_{3}
\end{array}\right)=\left(\begin{array}{ccc}
0&0&0\\
0&0&0\\
0&0&0
\end{array}\right)\left(\begin{array}{c}
e_{1}\\
e_{2}\\
e_{3}
\end{array}\right).
\end{align}
So $s^{0}=0.$
We see that
\begin{align}
\left\{\begin{array}{l}
De_{1}=-ce_{1},\\
De_{2}=-ce_{2},\\
De_{3}=-ce_{3}.\\
\end{array}\right.
\end{align}
By the analysis above, we have
\begin{align}
\left\{\begin{array}{l}
\alpha c=0,\\
\beta c=0,\\
\gamma c=0,\\
\delta c=0.\\
\end{array}\right.
\end{align}
On the basis of $\alpha+\delta\neq 0,$ $\alpha\gamma+\beta\delta=0,$ we get $c=0.$
\end{proof}

\begin{thm}
$(G_{5}, g, J)$ is the algebraic Schouten soliton associated to the connection $\nabla^{1}$ if and only if $c=0.$
\end{thm}
\begin{proof}
From
\begin{align}
\widetilde{{\rm Ric}}^{1}\left(\begin{array}{c}
e_{1}\\
e_{2}\\
e_{3}
\end{array}\right)=\left(\begin{array}{ccc}
0&0&0\\
0&0&0\\
0&0&0
\end{array}\right)\left(\begin{array}{c}
e_{1}\\
e_{2}\\
e_{3}
\end{array}\right),
\end{align}
we get $s^{1}=0.$
It follows that
\begin{align}
\left\{\begin{array}{l}
De_{1}=-ce_{1},\\
De_{2}=-ce_{2},\\
De_{3}=-ce_{3}.\\
\end{array}\right.
\end{align}
Thus,
\begin{align}
\left\{\begin{array}{l}
\alpha c=0,\\
\beta c=0,\\
\gamma c=0,\\
\delta c=0.\\
\end{array}\right.
\end{align}
Note that, $\alpha+\delta\neq 0,$ we get $c=0.$
\end{proof}

\begin{thm}
$(G_{6}, g, J)$ is the algebraic Schouten soliton associated to the connection $\nabla^{0}$ if and only if\\
(i) $\alpha+\delta\neq 0,$ $\beta=\gamma=0,$ $c=-\alpha^{2}+\alpha^{2}\lambda_{0},$\\
(ii) $\alpha\neq 0,$ $\beta^{2}=2\alpha^{2},$ $\gamma=\delta=0,$ $c=0.$
\end{thm}
\begin{proof}
We recall the following result:
\begin{align}
\widetilde{{\rm Ric}}^{0}\left(\begin{array}{c}
e_{1}\\
e_{2}\\
e_{3}
\end{array}\right)=\left(\begin{array}{ccc}
\frac{1}{2}\beta(\beta-\gamma)-\alpha^{2}&0&0\\
0&\frac{1}{2}\beta(\beta-\gamma)-\alpha^{2}&-\frac{1}{2}(-\gamma\alpha+\frac{1}{2}\delta(\beta-\gamma))\\
0&\frac{1}{2}(-\gamma\alpha+\frac{1}{2}\delta(\beta-\gamma))&0
\end{array}\right)\left(\begin{array}{c}
e_{1}\\
e_{2}\\
e_{3}
\end{array}\right).
\end{align}
And we get $s^{0}=\beta(\beta-\gamma)-2\alpha^{2}.$
Therefore for $(G_{6}, g, J)$ we have
\begin{align}
\left\{\begin{array}{l}
De_{1}=(\frac{1}{2}\beta(\beta-\gamma)-\alpha^{2}-(\beta(\beta-\gamma)-2\alpha^{2})\lambda_{0}-c)e_{1},\\
De_{2}=(\frac{1}{2}\beta(\beta-\gamma)-\alpha^{2}-(\beta(\beta-\gamma)-2\alpha^{2})\lambda_{0}-c)e_{2}-\frac{1}{2}(-\gamma\alpha+\frac{1}{2}\delta(\beta-\gamma))e_{3},\\
De_{3}=\frac{1}{2}(-\gamma\alpha+\frac{1}{2}\delta(\beta-\gamma))e_{2}-((\beta(\beta-\gamma)-2\alpha^{2})\lambda_{0}+c)e_{3}.\\
\end{array}\right.
\end{align}
By (4.14), we get
\begin{align}
\left\{\begin{array}{l}
\alpha(\frac{1}{2}\beta(\beta-\gamma)-\alpha^{2}-\beta(\beta-\gamma)\lambda_{0}+2\alpha^{2}\lambda_{0}-c)+\frac{1}{2}(\beta+\gamma)(\gamma\alpha-\frac{1}{2}\delta(\beta-\gamma))=0,\\
\beta(\beta(\beta-\gamma)-2\alpha^{2}-\beta(\beta-\gamma)\lambda_{0}+2\alpha^{2}\lambda_{0}-c)+\frac{1}{2}(\delta-\alpha)(\gamma\alpha-\frac{1}{2}\delta(\beta-\gamma))=0,\\
\gamma(\beta(\beta-\gamma)\lambda_{0}-2\alpha^{2}\lambda_{0}+c)-\frac{1}{2}(\delta-\alpha)(\gamma\alpha-\frac{1}{2}\delta(\beta-\gamma))=0,\\
\delta(\frac{1}{2}\beta(\beta-\gamma)-\alpha^{2}-\beta(\beta-\gamma)\lambda_{0}+2\alpha^{2}\lambda_{0}-c)-\frac{1}{2}(\beta+\gamma)(\gamma\alpha-\frac{1}{2}\delta(\beta-\gamma))=0.\\
\end{array}\right.
\end{align}
According to the condition $\alpha+\delta\neq 0,$ $\alpha\gamma-\beta\delta=0,$ we calculate that
\begin{align}
\left\{\begin{array}{l}
\frac{1}{2}\beta(\beta-\gamma)-\alpha^{2}-\beta(\beta-\gamma)\lambda_{0}+2\alpha^{2}\lambda_{0}-c=0,\\
(\beta+\gamma)(\gamma\alpha-\frac{1}{2}\delta(\beta-\gamma))=0,\\
(\beta+\gamma)(\beta(\beta-\gamma)\lambda_{0}-2\alpha^{2}\lambda_{0}+c)=0.\\
\end{array}\right.
\end{align}
Choose $\beta+\gamma=0,$ then we have $\beta(\alpha+\delta)=0,$ and $c=-\alpha^{2}+\alpha^{2}\lambda_{0}.$
Set $\beta+\gamma\neq0$ and $c=-\beta(\beta-\gamma)\lambda_{0}+2\alpha^{2}\lambda_{0}.$ By calculation, we have $\beta^{2}=2\alpha^{2},$ $\gamma=\delta=0$ and then $c=0.$
\end{proof}

\begin{thm}
$(G_{6}, g, J)$ is the algebraic Schouten soliton associated to the connection $\nabla^{1}$ if and only if\\
(i) $\alpha=\beta=0,$ $\delta\neq 0,$ $c=0,$\\
(ii) $\alpha\neq 0,$ $\beta=\gamma=0,$ $\alpha+\delta\neq 0,$ $c=-\alpha^{2}+2\alpha^{2}\lambda_{0}.$
\end{thm}
\begin{proof}
From \cite{Azami}, we get
\begin{align}
\widetilde{{\rm Ric}}^{1}\left(\begin{array}{c}
e_{1}\\
e_{2}\\
e_{3}
\end{array}\right)=\left(\begin{array}{ccc}
-(\alpha^{2}+\beta\gamma)&0&0\\
0&-\alpha^{2}&0\\
0&0&0
\end{array}\right)\left(\begin{array}{c}
e_{1}\\
e_{2}\\
e_{3}
\end{array}\right).
\end{align}
It is a simple matter to $s^{1}=-(2\alpha^{2}+\beta\gamma).$
It follows that
\begin{align}
\left\{\begin{array}{l}
De_{1}=-(\alpha^{2}+\beta\gamma-(2\alpha^{2}+\beta\gamma)\lambda_{0}+c)e_{1},\\
De_{2}=-(\alpha^{2}-(2\alpha^{2}+\beta\gamma)\lambda_{0}+c)e_{2},\\
De_{3}=-(-(2\alpha^{2}+\beta\gamma)\lambda_{0}+c)e_{3}.\\
\end{array}\right.
\end{align}
An easy computation shows that
\begin{align}
\left\{\begin{array}{l}
\alpha(\alpha^{2}+\beta\gamma-(2\alpha^{2}+\beta\gamma)\lambda_{0}+c)=0,\\
\beta(2\alpha^{2}+\beta\gamma-(2\alpha^{2}+\beta\gamma)\lambda_{0}+c)=0,\\
\gamma(\beta\gamma-(2\alpha^{2}+\beta\gamma)\lambda_{0}+c)=0,\\
\delta(\alpha^{2}+\beta\gamma-(2\alpha^{2}+\beta\gamma)\lambda_{0}+c)=0.\\
\end{array}\right.
\end{align}
The first and fourth equations of system Eq.(4.60) imply that
\begin{align}
(\alpha+\delta)(\alpha^{2}+\beta\gamma-(2\alpha^{2}+\beta\gamma)\lambda_{0}+c)=0.
\end{align}
Because $\alpha+\delta\neq 0,$ then we have $c=-\alpha^{2}-\beta\gamma+(2\alpha^{2}+\beta\gamma)\lambda_{0},$ $\alpha^{2}\beta=0,$ and $\alpha^{2}\gamma=0.$
Let $\alpha=0,$ then $\delta\neq 0,$ $\beta=0,$ $c=0.$
If $\alpha\neq 0,$ then $\beta=\gamma=0,$ $c=-\alpha^{2}-\beta\gamma+2\alpha^{2}\lambda_{0}.$
\end{proof}

\begin{thm}
$(G_{7}, g, J)$ is the algebraic Schouten soliton associated to the connection $\nabla^{0}$ if and only if\\
(i) $\alpha=\gamma=0,$ $\delta\neq 0,$ $c=0,$\\
(ii) $\alpha\neq 0,$ $\beta=\gamma=0,$ $\alpha+\delta\neq 0,$ $c=-\frac{1}{2}\alpha^{2}+2\alpha^{2}\lambda_{0}.$
\end{thm}
\begin{proof}
By \cite{Azami}, we have
\begin{align}
\widetilde{{\rm Ric}}^{0}\left(\begin{array}{c}
e_{1}\\
e_{2}\\
e_{3}
\end{array}\right)=\left(\begin{array}{ccc}
-(\alpha^2+\frac{1}{2}\beta\gamma)&0&\frac{1}{2}(\alpha\gamma+\frac{1}{2}\delta\gamma)\\
0&-(\alpha^2+\frac{1}{2}\beta\gamma)&-\frac{1}{2}(\alpha^2+\frac{1}{2}\beta\gamma)\\
-\frac{1}{2}(\alpha\gamma+\frac{1}{2}\delta\gamma)&\frac{1}{2}(\alpha^2+\frac{1}{2}\beta\gamma)&0
\end{array}\right)\left(\begin{array}{c}
e_{1}\\
e_{2}\\
e_{3}
\end{array}\right).
\end{align}
Clearly, $s^{0}=-(2\alpha^2+\beta\gamma).$
It follows that
\begin{align}
\left\{\begin{array}{l}
De_{1}=-(\alpha^2+\frac{1}{2}\beta\gamma-(2\alpha^2+\beta\gamma)\lambda_{0}+c)e_{1}+\frac{1}{2}(\alpha\gamma+\frac{1}{2}\delta\gamma)e_{3},\\
De_{2}=-(\alpha^2+\frac{1}{2}\beta\gamma-(2\alpha^2+\beta\gamma)\lambda_{0}+c)e_{2}-\frac{1}{2}(\alpha^2+\frac{1}{2}\beta\gamma)e_{3},\\
De_{3}=-\frac{1}{2}(\alpha\gamma+\frac{1}{2}\delta\gamma)e_{1}+\frac{1}{2}(\alpha^2+\frac{1}{2}\beta\gamma)e_{2}-(-(2\alpha^2+\beta\gamma)\lambda_{0}+c)e_{3}.\\
\end{array}\right.
\end{align}
A long but straightforward calculation shows that
\begin{align}
\left\{\begin{array}{l}
\alpha(\alpha^2+\frac{1}{2}\beta\gamma-(2\alpha^2+\beta\gamma)\lambda_{0}+c)-\frac{1}{2}(\beta+\gamma)(\alpha\gamma+\frac{1}{2}\gamma\delta)-\frac{1}{2}\alpha(\alpha^2+\frac{1}{2}\beta\gamma)=0,\\
\beta(\alpha^2+\frac{1}{2}\beta\gamma-(2\alpha^2+\beta\gamma)\lambda_{0}+c)-\frac{1}{2}\delta(\alpha\gamma+\frac{1}{2}\gamma\delta)=0,\\
\frac{1}{2}\alpha(\alpha^2+\frac{1}{2}\beta\gamma-2(2\alpha^2+\beta\gamma)\lambda_{0}+2c)-\frac{1}{2}\beta(\alpha\gamma+\frac{1}{2}\gamma\delta)=0,\\
\beta(\alpha^2+\frac{1}{2}\beta\gamma-(2\alpha^2+\beta\gamma)\lambda_{0}+c)=0,\\
\gamma(-(2\alpha^2+\beta\gamma)\lambda_{0}+c)-\frac{1}{2}\delta(\alpha\gamma+\frac{1}{2}\gamma\delta)=0,\\
\frac{1}{2}\delta(\alpha^2+\frac{1}{2}\beta\gamma-2(2\alpha^2+\beta\gamma)\lambda_{0}+2c)+\frac{1}{2}\beta(\alpha\gamma+\frac{1}{2}\gamma\delta)=0,\\
\frac{1}{2}\delta(\alpha^2+\frac{1}{2}\beta\gamma-2(2\alpha^2+\beta\gamma)\lambda_{0}+2c)+\frac{1}{2}(\beta+\gamma)(\alpha\gamma+\frac{1}{2}\gamma\delta)=0.\\
\end{array}\right.
\end{align}
The first and third equations of the system Eq.(4.64) yields
\begin{align}
\gamma(\alpha\gamma+\frac{1}{2}\gamma\delta)=0,
\end{align}
for $\alpha\gamma=0,$ we get
\begin{align}
\left\{\begin{array}{l}
\gamma\delta=0,\\
\gamma^{2}(\alpha+\frac{1}{2}\delta)=0.\\
\end{array}\right.
\end{align}
Let us regard $\gamma=0.$ We can get
\begin{align}
\left\{\begin{array}{l}
\alpha(\alpha^2-4\alpha^2\lambda_{0}+2c)=0,\\
\beta(\alpha^2-2\alpha^2\lambda_{0}+c)=0,\\
\delta(\alpha^2-4\alpha^2\lambda_{0}+2c)=0.\\
\end{array}\right.
\end{align}
Since $\alpha+\delta\neq 0,$ we have $\alpha^2-4\alpha^2\lambda_{0}+2c=0.$ Then
\begin{align}
\alpha^2\beta=0.
\end{align}
We assume that $\alpha=0,$ in this case, we obtain $\delta\neq 0,$ $c=0.$
If $\alpha\neq 0,$ then $\beta=0,$ $c=-\frac{1}{2}\alpha^{2}+2\alpha^{2}\lambda_{0}.$
\end{proof}

\begin{thm}
$(G_{7}, g, J)$ is the algebraic Schouten soliton associated to the connection $\nabla^{1}$ if and only if $\alpha\neq 0,$ $\beta=\gamma=0,$ $\delta=\frac{1}{2}\alpha,$ $c=-\frac{1}{2}\alpha^{2}+2\alpha^{2}\lambda_{0}.$
\end{thm}
\begin{proof}
From \cite{Azami}, we get
\begin{align}
\widetilde{{\rm Ric}}^{1}\left(\begin{array}{c}
e_{1}\\
e_{2}\\
e_{3}
\end{array}\right)=\left(\begin{array}{ccc}
-\alpha^{2}&\frac{1}{2}(\beta\delta-\alpha\beta)&-\beta(\alpha+\delta)\\
\frac{1}{2}(\beta\delta-\alpha\beta)&-(\alpha^{2}+\beta^{2}+\beta\gamma)&-\frac{1}{2}(\beta\gamma+\alpha\delta+2\delta^{2})\\
\beta(\alpha+\delta)&\frac{1}{2}(\beta\gamma+\alpha\delta+2\delta^{2})&0
\end{array}\right)\left(\begin{array}{c}
e_{1}\\
e_{2}\\
e_{3}
\end{array}\right).
\end{align}
Of course $s^{1}=-(2\alpha^{2}+\beta^{2}+\beta\gamma).$
It follows that
\begin{align}
\left\{\begin{array}{l}
De_{1}=-(\alpha^{2}-(2\alpha^{2}+\beta^{2}+\beta\gamma)\lambda_{0}+c)e_{1}+\frac{1}{2}(\beta\delta-\alpha\beta)e_{2}-\beta(\alpha+\delta)e_{3},\\
De_{2}=\frac{1}{2}(\beta\delta-\alpha\beta)e_{1}-(\alpha^{2}+\beta^{2}+\beta\gamma-(2\alpha^{2}+\beta^{2}+\beta\gamma)\lambda_{0}+c)e_{2}-\frac{1}{2}(\beta\gamma+\alpha\delta+2\delta^{2})e_{3},\\
De_{3}=\beta(\alpha+\delta)e_{1}+\frac{1}{2}(\beta\gamma+\alpha\delta+2\delta^{2})e_{2}-(-(2\alpha^{2}+\beta^{2}+\beta\gamma)\lambda_{0}+c)e_{3}.\\
\end{array}\right.
\end{align}
Therefore Equation (4.14) now becomes
\begin{align}
\left\{\begin{array}{l}
\alpha(\alpha^{2}+\beta^{2}+\beta\gamma-(2\alpha^{2}+\beta^{2}+\beta\gamma)\lambda_{0}+c)+(\beta^{2}+\beta\gamma)(\alpha+\delta)+\frac{1}{2}\beta(\beta\delta-\alpha\beta)\\
-\frac{1}{2}\alpha(\beta\gamma+\alpha\delta+2\delta^2)=0,\\
\beta(\frac{1}{2}\alpha^{2}-(2\alpha^{2}+\beta^{2}+\beta\gamma)\lambda_{0}+c+\frac{3}{2}\alpha\delta+\delta^{2})=0,\\
\beta(2\alpha^{2}+\beta^{2}+\beta\gamma-(2\alpha^{2}+\beta^{2}+\beta\gamma)\lambda_{0}+c)-\beta(\beta\gamma+\alpha\delta+2\delta^{2})+\beta(\alpha+\delta)(\delta-\alpha)=0,\\
\alpha(-(2\alpha^{2}+\beta^{2}+\beta\gamma)\lambda_{0}+c)+\frac{1}{2}\alpha(\beta\gamma+\alpha\delta+2\delta^{2})+\frac{1}{2}(\beta-\gamma)(\beta\delta-\alpha\beta)+\beta^{2}(\alpha+\delta)=0,\\
-\beta(\beta^{2}+\beta\gamma+(2\alpha^{2}+\beta^{2}+\beta\gamma)\lambda_{0}-c)+\frac{1}{2}(\alpha-\delta)(\beta\delta-\alpha\beta)+\beta(\beta\gamma+\alpha\delta+2\delta^{2})=0,\\
\beta(-\frac{1}{2}\alpha\delta-\frac{1}{2}\delta^{2}-(2\alpha^{2}+\beta^{2}+\beta\gamma)\lambda_{0}+c)=0,\\
\gamma(\beta^{2}+\beta\gamma-(2\alpha^{2}+\beta^{2}+\beta\gamma)\lambda_{0}+c)-\frac{1}{2}(\alpha-\delta)(\beta\delta-\alpha\beta)+\beta(\alpha+\delta)(\delta-\alpha)=0,\\
\delta(-(2\alpha^{2}+\beta^{2}+\beta\gamma)\lambda_{0}+c)-\frac{1}{2}(\beta-\gamma)(\beta\delta-\alpha\beta)+\frac{1}{2}\delta(\beta\gamma+\alpha\delta+2\delta^{2})-\beta^{2}(\alpha+\delta)=0,\\
-\delta(\alpha^{2}+\beta^{2}+\beta\gamma-(2\alpha^{2}+\beta^{2}+\beta\gamma)\lambda_{0}+c)+\frac{1}{2}\beta(\beta\delta-\alpha\beta)+\frac{1}{2}\delta(\beta\gamma+\alpha\delta+2\delta^{2})\\
+(\beta^{2}+\beta\gamma)(\alpha+\delta)=0.\\
\end{array}\right.
\end{align}
Throughout the proof recall that $\alpha+\delta\neq 0$ and $\alpha\gamma=0.$
Assume first that $\alpha\neq 0,$ $\gamma=0.$ In this case,
\begin{align}
\left\{\begin{array}{l}
\alpha(\alpha^{2}+\beta^{2}-(2\alpha^{2}+\beta^{2})\lambda_{0}+c)+\beta^{2}(\alpha+\delta)+\frac{1}{2}\beta(\beta\delta-\alpha\beta)-\frac{1}{2}\alpha(\alpha\delta+2\delta^2)=0,\\
\beta(\frac{1}{2}\alpha^{2}-(2\alpha^{2}+\beta^{2})\lambda_{0}+c+\frac{3}{2}\alpha\delta+\delta^{2})=0,\\
\beta(2\alpha^{2}+\beta^{2}-(2\alpha^{2}+\beta^{2})\lambda_{0}+c)-\beta(\alpha\delta+2\delta^{2})+\beta(\alpha+\delta)(\delta-\alpha)=0,\\
\alpha(-(2\alpha^{2}+\beta^{2})\lambda_{0}+c)+\frac{1}{2}\alpha(\alpha\delta+2\delta^{2})+\frac{1}{2}\beta(\beta\delta-\alpha\beta)+\beta^{2}(\alpha+\delta)=0,\\
-\beta(\beta^{2}+(2\alpha^{2}+\beta^{2})\lambda_{0}-c)+\frac{1}{2}(\alpha-\delta)(\beta\delta-\alpha\beta)+\beta(\alpha\delta+2\delta^{2})=0,\\
\beta(-\frac{1}{2}\alpha\delta-\frac{1}{2}\delta^{2}-(2\alpha^{2}+\beta^{2})\lambda_{0}+c)=0,\\
-\frac{1}{2}(\alpha-\delta)(\beta\delta-\alpha\beta)+\beta(\alpha+\delta)(\delta-\alpha)=0,\\
\delta(-(2\alpha^{2}+\beta^{2})\lambda_{0}+c)-\frac{1}{2}\beta(\beta\delta-\alpha\beta)+\frac{1}{2}\delta(\alpha\delta+2\delta^{2})-\beta^{2}(\alpha+\delta)=0,\\
-\delta(\alpha^{2}+\beta^{2}-(2\alpha^{2}+\beta^{2})\lambda_{0}+c)+\frac{1}{2}\beta(\beta\delta-\alpha\beta)+\frac{1}{2}\delta(\alpha\delta+2\delta^{2})+\beta^{2}(\alpha+\delta)=0.\\
\end{array}\right.
\end{align}
Next suppose that $\beta=0,$
\begin{align}
\left\{\begin{array}{l}
\alpha(\alpha^{2}-2\alpha^{2}\lambda_{0}+c)-\frac{1}{2}\alpha(\alpha\delta+2\delta^2)=0,\\
\alpha(-2\alpha^{2}\lambda_{0}+c)+\frac{1}{2}\alpha(\alpha\delta+2\delta^{2})=0,\\
\delta(-2\alpha^{2}\lambda_{0}+c)+\frac{1}{2}\delta(\alpha\delta+2\delta^{2})=0,\\
-\delta(\alpha^{2}-2\alpha^{2}\lambda_{0}+c)+\frac{1}{2}\delta(\alpha\delta+2\delta^{2})=0.\\
\end{array}\right.
\end{align}
Then we get
\begin{align}
(\alpha+\delta)(2\delta-\alpha)=0,
\end{align}
that is, $\delta=\frac{1}{2}\alpha,$ $c=-\frac{1}{2}\alpha^{2}+2\alpha^{2}\lambda_{0}.$
\end{proof}

\vskip 1 true cm

\section{Conclusion}
The main work of this paper is to investigate algebraic Schouten solitons associated to Levi-Civita connections, canonical connections and Kobayashi-Nomizu connections, and classify algebraic Schouten solitons associated to Levi-Civita connections, canonical connections and Kobayashi-Nomizu connections on three-dimensional Lorentzian Lie groups with the product structure.\\

\vskip 1 true cm

\section{Declarations}
Ethics approval and consent to participate No applicable.\\

Consent for publication No applicable.\\

Availability of data and material The authors confirm that the data supporting the findings of this study are available within the article.\\

Competing interests The authors declare no conflict of interest.\\

Funding This research was funded by National Natural Science Foundation of China: No.11771070.\\

Authors' contributions S.L. studies conceptualization and writing (review and editing) the manuscript, Y.W. funding acquisition and project administration.\\


\bigskip
\bigskip

\noindent {\footnotesize {\it S. Liu} \\
{School of Mathematics and Statistics, Northeast Normal University, Changchun 130024, China}\\
{Email: liusy719@nenu.edu.cn}

\noindent {\footnotesize {\it Y. Wang} \\
{School of Mathematics and Statistics, Northeast Normal University, Changchun 130024, China}\\
{Email: wangy581@nenu.edu.cn}


\begin{thebibliography}{20}

\bibitem{Lauret} Lauret, J., Ricci soliton homogeneous nilmanifolds, {\it Math. Ann.}, {\bf 319}(4), 2001, 715-733.

\bibitem{Onda} Onda, K., Examples of algebraic Ricci solitons in the pseudo-Riemannian case, {\it Acta Math. Hungar.}, {\bf 144}(1), 2014, 247-265.

\bibitem{Batat} Batat, W., Onda, K., Algebraic Ricci solitons of three-dimensional Lorentzian Lie groups, {\it J. Geom. Phys.}, {\bf 114}, 2017, 138-152.

\bibitem{Etayo} Etayo, F., Santamaria, R., Distinguished connections on $(J^2=\pm 1)$-metric manifolds, {\it Arch. Math.}, {\bf 52}(3), 2016, 159-203.

\bibitem{Wang1} Wang, Y., Canonical connections and algebraic Ricci solitons of three-dimensional Lorentzian Lie groups, {\it Chin. Ann. Math. Ser. B.}, {\bf 43}(3), 2022, 443-458.

\bibitem{Brozos} Brozos-Vazquez, M., Garcia-Rio, E., Gavino-Fernandez, S., Locally conformally flat Lorentzian gradient Ricci solitons, {\it J. Geom. Anal.}, {\bf 23}, 2013, 1196-1212.

\bibitem{Azami} Azami, S., Generalized Ricci solitons of three-dimensional Lorentzian Lie groups associated canonical connections and Kobayashi-Nomizu connections, {\it J. Nonlinear Math. Phys.}, https://doi.org/10.1007/s44198-022-00069-2.

\bibitem{Wang2} Wang, Y., Affine Ricci soliton of three-dimensional Lorentzian Lie groups, {\it J. Nonlinear Math. Phys.}, {\bf 28}(3), 2021, 277-291.

\bibitem{Wang3} Wu, T., Wang, Y., Affine Ricci solitons associated to the Bott connection on three-dimensional Lorentzian Lie groups, {\it Turkish J. Math.}, {\bf 45}, 2021, 2773-2816.

\bibitem{Wears} Wears, T.H., On algebraic solitons for geometric evolution equations on three-dimensional Lie groups, {\it Tbilisi Math. J.}, {\bf 9}(2), 2016, 33-58.

\bibitem{Louzao} Calvino-Louzao, E., Hervella, L.M., Seoane-Bascoy, J., Vazquez-Lorenzo, R., Homogeneous Cotton solitons, {\it J. Phys. A: Math. Theor.}, {\bf 46}(28), 2013, 285204.

\bibitem{Milnor} Milnor, J., Curvature of left invariant metrics on Lie groups, {\it Adv. Math.}, {\bf 21}, 1976, 293-329.

\bibitem{Rahmani} Rahmani, S., M\'{e}triques de Lorentz sur les groupes de Lie unimodulaires de dimension trois, {\it J. Geom. Phys.}, {\bf 9}, 1992, 295-302.

\bibitem{Cordero} Cordero, L.A., Parker, P.E., Left-invariant Lorentzian metrics on 3-dimensional Lie groups, {\it Rend. Mat. Appl.}, {\bf 17}(7), 1997, 129-155.

\bibitem{Calvaruso} Calvaruso, G., Homogeneous structures on three-dimensional homogeneous Lorentzian manifolds, {\it J. Geom. Phys.}, {\bf 57}(4), 2007, 1279-1291.

\bibitem{Salimi} Salimi, H.R., On the geometry of some para-hypercomplex Lie groups, {\it Arch. Math.}, {\bf 45}(3), 2009, 159-170.





\end{thebibliography}
\end{document}